\documentclass[10pt]{amsart}
\usepackage{amsmath,amssymb,mathrsfs,hyperref,color}
\usepackage{subfigure}
\usepackage{float} 
\usepackage{times}
\usepackage{tikz}
\usetikzlibrary{intersections}
\usepackage{paralist}

\makeatletter
\def\setliststart#1{\setcounter{\@listctr}{#1}%
  \addtocounter{\@listctr}{-1}}
\makeatother

\makeatletter
\@addtoreset{figure}{section}
\makeatother

\usepackage{calc}
 \newtheorem{The}{Theorem}[section]
 \newtheorem{Cor}[The]{Corollary}
 \newtheorem{Lem}[The]{Lemma}
 \newtheorem{Pro}[The]{Proposition}
 \theoremstyle{definition}
 
 \theoremstyle{remark}
 \newtheorem{Rem}[The]{Remark}
 
 \numberwithin{equation}{section}

\newcommand{\R}{\mathbb{R}}

\title[Representation formulas for contact type Hamilton-Jacobi equations]{Representation formulas for  contact type Hamilton-Jacobi equations}
\author{Jiahui Hong \and Wei Cheng \and Shengqing Hu \and Kai Zhao}
\address{Department of Mathematics, Nanjing University, Nanjing 210093, China}
\email{623006874@qq.com}
\address{Department of Mathematics, Nanjing University, Nanjing 210093, China}
\email{chengwei@nju.edu.cn}
\address{Department of Mathematics, Nanjing University, Nanjing 210093, China}
\email{shengqinghu@nju.edu.cn}
\address{Department of Mathematics, Nanjing University, Nanjing 210093, China}
\email{15251879427@163.com}

\date{\today}
\begin{document}
\maketitle

\begin{abstract}
	We discuss various kinds of representation formulas for the viscosity solutions of the contact type Hamilton-Jacobi equations by using the Herglotz' variational principle.
\end{abstract}

\section{Introduction}
Let $M$ be a $C^1$ connected and compact manifold without boundary. Let $TM$ and $T^*M$ denote the tangent and cotangent bundles respectively. A point of $TM$ will be denoted by $(x,v)$ with $x\in M$ and $v\in T_xM$, and a point of $T^*M$ by $(x, p)$ with $p\in  T_x^*M$ is a linear form on the vector space $T_xM$. With a slight abuse of notation, we shall denote by $|\cdot|_x$ the norm on the fiber $T_xM$ and also the dual norm on $T_x^*M$.

In this paper, we want to discuss the representation formula for the viscosity solutions of the evolutionary Hamilton-Jacobi equation
\begin{equation}\label{eq:HJe}\tag{HJ$_e$}
\begin{cases}
	D_tu(t,x)+H(t,x,u(t,x),D_xu(t,x))=0,& (t,x)\in(0,+\infty)\times M\\
	u(0,x)=\phi(x),& x\in M,
\end{cases}
\end{equation}
and the stationary equation
\begin{equation}\label{eq:intro_HJs}\tag{HJ$_s$}
H(x,u(x),Du(x))=0,\quad x\in M.
\end{equation}
Here we suppose $0$ on the right side of \eqref{eq:intro_HJs} belongs to the set of Ma\~n\'e's critical values.

Because of the Lagrangian formalism, we endow some suitable conditions on the associated Lagrangian $L$ with respect to a convex $H$ defined by 
\begin{align*}
	L(s,x,v,u)=\sup_{p\in T_x^*M}\{p\cdot v -H(s,x,p,u)\},\quad (s,x,v,u)\in \R\times TM\times\R.
\end{align*}
The conditions on $L$ are imposed at the beginning of Section 2.

The representation formula for the viscosity solutions of the Hamilton-Jacobi equations in various kind problems typically connects the solution of PDEs to the value function for the relevant problems from calculus of variations and optimal control. A representation formula provides further information on the underlying dynamical systems which is important for certain finer analysis of the solutions including qualitative Lipschitz and semiconcavity estimate and some dynamical systems implications. For classical convex Hamiltonians, see, for instance, \cite{Lions_book,Ishii1985,Ishii1988,Bardi_Capuzzo-Dolcetta1997,Cannarsa_Sinestrari_book,Rampazzo2005,Barles_book} and \cite{Ishii_Mitake2007,Fathi_Siconolfi2004,Fathi_Siconolfi2005,Mitake2008,Castelpietra_Rifford2010,Mantegazza_Mennucci2003} for weak KAM and geometric aspects .

The representation formula for the viscosity solutions of \eqref{eq:HJe} and \eqref{eq:intro_HJs} is known for the discounted systems even in the early period of the theory of viscosity solutions. A systematical approach of equations \eqref{eq:HJe} and \eqref{eq:intro_HJs} in general firstly appears in \cite{Wang_Wang_Yan2017,Wang_Wang_Yan2019_1} in an implicit way. An alternative Lagrangian approach is based on a rigorous treatment of the classical Herglotz' variational principle (\cite{CCWY2018,CCJWY2019}). Recall some basic results from \cite{CCJWY2019,CCWY2018} on the Herglotz' variational principle and the Hamilton-Jacobi equations of contact type. 

For any $t_2>t_1$, $u_0\in\R$ and $x,y\in M$, denote the set
\begin{align*}
	\Gamma^{t_1,t_2}_{x,y}=\{\xi\in W^{1,1}([t_1,t_2],M ): \xi(t_1)=x,\xi(t_2)=y\},
\end{align*}
and consider the following Carath\'eodory equation
\begin{equation}\label{eq:caratheodory_L}
\left\{
\begin{aligned}
&\dot{u}_{\xi}(s)=L(s,\xi(s),\dot{\xi}(s),u_\xi(s))  \quad a.e.\ s\in[t_1,t_2]\\
&u_{\xi}(t_1)=u_0.\\
\end{aligned}
\right.
\end{equation}
It is clear that equation \eqref{eq:caratheodory_L} admits a unique solution (see \cite{Coddington_Levinson_book}). We define
\begin{equation}\label{eq:fundamental_solution}
	h_L(t_1,t_2,x,y,u_0):=\inf_{\xi\in\Gamma^{t_1,t_2}_{x,y}}\int^{t_2}_{t_1}L(s,\xi(s),\dot{\xi}(s),u_{\xi}(s))\ ds.
\end{equation}
As showed in \cite{CCJWY2019}, the infimum in the definition of the \emph{negative-type fundamental solution} $h_L(t_1,t_2,x,y,u_0)$ can be achieved and any minimizer $\xi\in\Gamma^{t_1,t_2}_{x,y}$ is as smooth as $L$. We introduce the associated Lax-Oleinik operator
\begin{align*}
	(\mathbf{T}^{t_2}_{t_1}\phi)(x)=\inf_{y\in\R^n}\{\phi(y)+h_L(t_1,t_2,y,x,\phi(y))\},\quad \quad t_2>t_1, x\in M,
\end{align*}
where $\phi:M\to[-\infty,+\infty]$ is any function. For any $t>0$ and $x\in M$, set
\begin{equation}\label{eq:viscosity_solution}
	u(t,x):=(\mathbf{T}^{t}_{0}\phi)(x)=\inf_{y\in\R^n}\{\phi(y)+h_L(0,t,y,x,\phi(y))\}.
\end{equation}
It is known that $u(t,x)$ defined in \eqref{eq:viscosity_solution} is a viscosity solution of \eqref{eq:HJe} (see Proposition \ref{facts_viscosity_soluiton} for a precise statement). 

Comparing to the implicit representation formula in \cite{Wang_Wang_Yan2019_1}, an advantage of Herglotz' variational principle is that, one can obtain various kind of representation formulas by choosing different ways to solve the Carath\'eodory equation \eqref{eq:caratheodory_L}. These representation formulas are also useful for many applications. One example is the problem on vanishing discount (\cite{DFIZ2016}) and vanishing contact structure (\cite{Zhao_Cheng2018,CCIZ2019}), where such a representation formula play an important role. Another example is the problem on the propagation of singularities. The regularity properties of the fundamental solution such as quantitative semiconcavity and convexity estimates can be obtained by using Herglotz' variational principle, as well as the representation formulas for both negative-type and positive-type fundamental solutions. which will be adapted to our intrinsic method developed in \cite{Cannarsa_Cheng3,Cannarsa_Cheng_Fathi2017,CCMW2018,Cannarsa_Cheng_Fathi2018}.

The paper is organized as follows. We discuss the representation formulas in Section 2 and Section 3 for evolutionary equation \eqref{eq:HJe} and stationary equation \eqref{eq:intro_HJs} respectively. The last section contains some concluding remarks including the discounted systems comparing to some known results.

\medskip

\noindent\textbf{Acknowledgements.} Wei Cheng is partly supported by Natural Scientific Foundation of China (Grant No. 11871267, 11631006 and 11790272). The authors thank for Qinbo Chen for helpful discussion.

\section{representation formula of evolutionary equation}

We assume that $L$ is of class $C^1$. For the purpose of this paper, we need the following conditions:
\begin{enumerate}[(L1)]
    \item $L(s,x,\cdot,u)$  is strict convex on $T_xM$ for all $s\in\R$, $x\in M$ and $u\in \mathbb{R}$.
    \item  There exist $c_0 > 0$ and a superlinear and nondecreasing function $\theta_0:[0,+\infty)\to [0,+\infty)$, such that
	\begin{align*}
		L(s,x,v,0)\geq \theta_0(|v|_x)-c_0,\quad (s,x,v)\in\R\times TM.
	\end{align*}
	\item There exists $K>0$ such that
	\begin{align*}
		|L_u(s,x,v,u)|\leqslant K,\quad (s,x,v,u)\in\R\times TM\times\R.
	\end{align*}
	\item There exists $C_1,C_2>0$ such that $|L_t(s,x,v,u)|\leqslant C_1+C_2L(s,x,v,u)$ for all $(s,x,v,u)\in\R\times TM\times\R$.
	\item The map $u\mapsto L(s,x,v,u)$ is concave for all $(s,x,v)\in\R\times TM$.
	\item $L_u(s,x,v,u)<0$ for all $(s,x,v,u)\in\R\times TM\times\R$.
\end{enumerate}

\subsection{Herglotz' variational principle}

We recall some known results based on Herglotz' variational principle and the representation of the viscosity solutions. Set
\begin{align*}
	\mathcal{A}_{t,x}=\{\xi\in W^{1,1}([0,t],M): \xi(t)=x\}.
\end{align*}
We suppose condition (L1)-(L4) is satisfied.

\begin{Pro}[\cite{CCJWY2019}]\label{facts_viscosity_soluiton}
Let $\phi$ be lower semi-continuous and $(\kappa_1,\kappa_2)$-Lipschitz in the large\footnote{Let $(X,d)$ be a metric space. A function $\phi:X\to\R$ is called $(\kappa_1,\kappa_2)$-Lipschitz in the large if there exists $\kappa_1,\kappa_2\geqslant 0$ such that $|\phi(y)-\phi(x)|\leqslant\kappa_1+\kappa_2d(x,y)$, for all $x,y\in X$} and the function $u(t,x)$ be defined in \eqref{eq:viscosity_solution}.
\begin{enumerate}[\rm (1)]
	\item The function $u(t,x)$ is finite-valued.
	\item For any $t>0$ and $x\in M$ the function $y\mapsto \phi(y)+h_L(0,t,y,x,\phi(y))$ admits a minimizer.
	\item For any $t>0$ and $x\in M$, let $y_{t,x}$ be a minimizer of the function $y\mapsto \phi(y)+h_L(0,t,y,x,\phi(y))$. Then, there exists a minimizer $\xi_{t,x}\in\Gamma^{0,t}_{y_{t,x},x}$ for $h_L(0,t,y_{t,x},x,\phi(y_{t,x}))$ such that
	\begin{align*}
		u(t,x)=\phi(y_{t,x})+\int^{t}_{0}L(s,\xi_{t,x}(s),\dot{\xi}_{t,x}(s),u_{\xi_{t,x}}(s))\ ds.
	\end{align*}
	\item Equivalently, for any $t>0$ and $x\in M$, there exists $\xi\in\mathcal{A}_{t,x}$ such that
	\begin{align*}
	    u(t,x)=\phi(\xi(0))+\int^{t}_{0}L(s,\xi(s),\dot{\xi}(s),u_{\xi}(s))\ ds.
	\end{align*}
	\item Moreover, $u(\cdot,x)$ is right continuous at $t=0$ for all $x\in M$, and the extension of $u$ on $[0,+\infty)\times M$ is a unique solution of \eqref{eq:HJe} in the sense of viscosity.
\end{enumerate}
\end{Pro}

\subsection{Representation formula in general} 

The representation formula for the viscosity solutions of Hamilton-Jacobi equation \eqref{eq:HJe} and \eqref{eq:intro_HJs} was first systematically studied in the papers \cite{Wang_Wang_Yan2017,Wang_Wang_Yan2019_1} by using an implicit variational principle for $M$ being compact and $L$ being time-independent. Equivalently, by using Herglotz' variational principle (see \cite{CCJWY2019,CCWY2018}), we have our first representation formula.

\begin{Pro}[Representation formula I]
Suppose $L$ satisfies condition \mbox{\rm (L1)-(L4)} and $H$ is the associated Hamiltonian. If $\phi$ is lower semi-continuous and $(\kappa_1,\kappa_2)$-Lipschitz in the large, then the unique viscosity solution $u$ of \eqref{eq:HJe} has the following representation: for any $t>0$ and $x\in M$,
\begin{equation}\label{eq:rep1}
	u(t,x)=\inf_{\xi\in\mathcal{A}_{t,x}} u_\xi(t) =\inf_{\xi\in\mathcal{A}_{t,x}}\left\{\phi(\xi(0))+\int^t_0L(s,\xi(s),\dot{\xi}(s),u_{\xi}(s))\ ds\right\},
\end{equation}
where $u_{\xi}$ is uniquely determined by \eqref{eq:caratheodory_L} with $u_0=\phi(\xi(0))$.
\end{Pro}

The second representation formula for the viscosity solutions of Hamilton-Jacobi equation \eqref{eq:HJe} appears in \cite{Zhao_Cheng2018}. Here we give a slight extension for the time-dependent Lagrangian on manifold, with a different proof.

\begin{Pro}[Representation formula II]\label{RF_II}
Suppose $L$ satisfies condition \mbox{\rm (L1)-(L4)} and $H$ is the associated Hamiltonian. If $\phi$ is lower semi-continuous and $(\kappa_1,\kappa_2)$-Lipschitz in the large, then the unique viscosity solution $u$ of \eqref{eq:HJe} has the following representation: for any $t>0$ and $x\in M$,
\begin{equation}\label{eq:rep2}
	u(t,x)=\inf_{\xi\in\mathcal{A}_{t,x}}\left\{e^{\int_0^t L_u d\tau}\phi(\xi(0))+\int^t_0e^{\int_s^t L_u d\tau}(L-u_{\xi}L_u)\ ds\right\},
\end{equation}
where $L_u(s)=L_u(s,\xi(s),\dot{\xi}(s),u_{\xi}(s))$ and $u_{\xi}$ is uniquely determined by \eqref{eq:caratheodory_L} with $u_0=\phi(\xi(0))$.
\end{Pro}

\begin{proof}
Adding a term $-u_{\xi}(s)L_u$ to the both sides of \eqref{eq:caratheodory_L}, we obtain
\begin{align*}
	\dot{u}_{\xi}-u_{\xi}L_u(s,\xi,\dot{\xi},u_{\xi})=L(s,\xi,\dot{\xi},u_{\xi})-u_{\xi}L_u(s,\xi,\dot{\xi},u_{\xi})
\end{align*}
which leads to
\begin{align*}
	\frac{d}{ds}\left(e^{-\int_0^s L_ud\tau} u_{\xi} \right)=e^{-\int_0^s L_u d\tau} (L-u_{\xi}L_u),
\end{align*}
Thus \eqref{eq:rep2} follows by integrating both sides from $0$ to $t$. Due to the relation $u(t,x)=\inf_{\xi\in\mathcal{A}_{t,x}}u_{\xi}(t)$, this completes the proof.
\end{proof}

\subsection{Representation formula for $u$-concave Lagrangian}

Fix any $t_2>t_1$, $u_0\in\R$ and $x,y\in M$. Let $\xi^*\in\Gamma^{t_1,t_2}_{x,y}$ be a minimizer for the functional $u_{\xi}(t)$ where $u_{\xi}$ is uniquely determined by \eqref{eq:caratheodory_L}. Consider a new Carath\'eodory equation
 \begin{equation}\label{eq:Cara2}
	\begin{cases}
		\dot{v}_{\eta}=L(s,\eta,\dot{\eta},u_{\xi^*})+L_u(s,\eta,\dot{\eta},u_{\xi^*})(v_{\eta}-u_{\xi^*}),\quad a.e., s\in[t_1,t_2]\\
		v_{\eta}(t_1)=u_0,
	\end{cases}
\end{equation}
where $v_{\eta}(s)=v_{\xi^*,\eta}(s)$ for $\eta\in\Gamma^{t_1,t_2}_{x,y}$.

\begin{Lem}\label{equiv1}
Suppose $L$ satisfies condition \mbox{\rm (L1)-(L5)} and $H$ is the associated Hamiltonian. If $u_{\xi}$ and $v_{\eta}=v_{\xi^*,\eta}$ are determined by \eqref{eq:caratheodory_L} and \eqref{eq:Cara2} for $\xi,\eta\in\Gamma^{t_1,t_2}_{x,y}$ respectively, then, 
\begin{enumerate}[\rm (1)]
	\item we have
	\begin{align*}
	    \inf_{\xi\in\Gamma^{t_1,t_2}_{x,y}}u_{\xi}(t_2)=\inf_{\eta\in\Gamma^{t_1,t_2}_{x,y}}v_{\eta}(t_2);
	\end{align*}
	\item $\xi^*\in\arg\min_{\eta\in\Gamma^{t_1,t_2}_{x,y}}v_{\xi^*,\eta}(t_2)\subset\arg\min_{\xi\in\Gamma^{t_1,t_2}_{x,y}}u_\xi(t_2)$. In particular, if $\xi^*$ is a unique minimizer for $h_L(t_1,t_2,x,y,u_0)$, then the relation of inclusion is indeed an equality and each of two sets is a singleton.
\end{enumerate} 	
\end{Lem}

\begin{proof}
For any $\xi\in\Gamma^{t_1,t_2}_{x,y}$, set $w_{\xi}=u_{\xi}-v_{\xi}$. Then, by concavity of $L$ with respect to $u$ we have that
\begin{align*}
	\dot{w}_{\xi}=&\,L(s,\xi,\dot{\xi},u_{\xi})-[L(s,\xi,\dot{\xi},u_{\xi^*})+L_u(s,\xi,\dot{\xi},u_{\xi^*})(v_{\xi}-u_{\xi^*})]\\
	\leqslant&\,[L_u(s,\xi,\dot{\xi},u_{\xi^*})(u_{\xi}-u_{\xi^*})]-[L_u(s,\xi,\dot{\xi},u_{\xi^*})(v_{\xi}-u_{\xi^*})]\\
	=&\,L_u(s,\xi,\dot{\xi},u_{\xi^*})w_{\xi}
\end{align*}
with $w_{\xi}(t_1)=0$. It follows that $w_{\xi}(s)\leqslant0$ for all $s\in[t_1,t_2]$. Therefore, $u_{\xi}(t_2)\leqslant v_{\xi}(t_2)$ for all $\xi\in\Gamma^{t_1,t_2}_{x,y}$. Thus,
\begin{equation}\label{eq:leq}
	\inf_{\xi\in\Gamma^{t_1,t_2}_{x,y}}u_{\xi}(t_2)\leqslant \inf_{\eta\in\Gamma^{t_1,t_2}_{x,y}}v_{\eta}(t_2).
\end{equation}

Now, set $\xi=\eta=\xi^*$ in \eqref{eq:caratheodory_L} and \eqref{eq:Cara2} respectively and $w=v_{\xi^*}-u_{\xi^*}$. Then
\begin{align*}
	\begin{cases}
		\dot{w}(s)=L_u(s,\xi^*(s),\dot{\xi}^*(s),u_{\xi^*}(s))w(s),\quad a.e.\ s\in[t_1,t_2]\\
		w(t_1)=0.
	\end{cases}
\end{align*}
This implies $w\equiv0$ on $[t_1,t_2]$. It follows that
\begin{equation}\label{eq:leq2}
	\inf_{\eta\in\Gamma^{t_1,t_2}_{x,y}}v_{\eta}(t_2)\leqslant v_{\xi^*}(t_2)=u_{\xi^*}(t)=\inf_{\xi\in\Gamma^{t_1,t_2}_{x,y}}u_{\xi}(t_2).
\end{equation}
This completes the proof of (1) together with \eqref{eq:leq}.

To see (2), we suppose $\eta'\in\arg\min_{\eta\in\Gamma^{t_1,t_2}_{x,y}}v_\eta(t_2)$. Then
\begin{align*}
	u_{\eta'}(t_2)\leqslant v_{\eta'}(t_2)=v_{\xi^*}(t_2)=u_{\xi^*}(t)=\inf_{\xi\in\Gamma^{t_1,t_2}_{x,y}}u_{\xi}(t_2)
\end{align*}
by \eqref{eq:leq2}. This implies $\eta'\in\arg\min_{\xi\in\Gamma^{t_1,t_2}_{x,y}}u_\xi(t_2)$. The combination of \eqref{eq:leq} and \eqref{eq:leq2} leads to the inclusion $\xi^*\in\arg\min_{\eta\in\Gamma^{t_1,t_2}_{x,y}}v_{\xi^*,\eta}(t_2)$.
\end{proof}

Fix $t_2>t_1$, $u_0\in\R$ and $x,y\in M$. Let $\xi^*$ be a minimizer for $h_L(t_1,t_2,x,y,u_0)$. In light of Lemma \ref{equiv1}, we define a new Lagrangian
\begin{equation}\label{eq:L_xi}
	\begin{split}
		L^{\xi^*}(s,x,v,u):=&\,L(s,x,v,u_{\xi^*}(s))+L_u(s,x,v,u_{\xi^*}(s))(u-u_{\xi^*}(s)).	\end{split}
\end{equation}

Now we can reformulate our results in Lemma \ref{equiv1}.

\begin{Pro}\label{equiv2}
Fix $t_2>t_1$, $u_0\in\R$ and $x,y\in M$. Let $\xi^*$ be a minimizer for $h_L(t_1,t_2,x,y,u_0)$. Then,
\begin{enumerate}[\rm (1)]
	\item $h_L(t_1,t_2,x,y,u_0)=h_{L^{\xi^*}}(t_1,t_2,x,y,u_0)$;
	\item If $\xi'$ is a minimizer for $h_{L^{\xi^*}}(t_1,t_2,x,y,u_0)$, then $\xi'$ is a minimizer for $h_L(t_1,t_2,x,y,u_0)$.
\end{enumerate}
\end{Pro}



\begin{The}[Representation formula III]\label{rep3}
Suppose $L$ satisfies condition \mbox{\rm (L1)-(L5)} and $H$ is the associated Hamiltonian. If $\phi$ is lower semi-continuous and $(\kappa_1,\kappa_2)$-Lipschitz in the large, then the unique viscosity solution $u$ of \eqref{eq:HJe} has the following representation: for any $t>0$ and $x\in M$, if $\xi^*\in\mathcal{A}_{t,x}$ be a minimal curve in the definition of $u(t,x)$ with $y^*=\xi^*(0)$, then
\begin{equation}\label{eq:rep3}
	u(t,x)=\inf_{\eta\in\Gamma^{0,t}_{y^*,x}}\left\{e^{\int^t_0L_ud\tau}\phi(y^*)+\int^t_0e^{\int^t_sL_ud\tau}\{L(s,\eta,\dot{\eta},u_{\xi^*})-u_{\xi^*}L_u\}\ ds\right\},
\end{equation}
where $L_u(s):=L_u(s,\eta(s),\dot{\eta}(s),u_{\xi^*}(s))$. Moreover, the right side of \eqref{eq:rep3} is independent of the choice of $\xi^*$.
\end{The}

\begin{proof}
Let $t>0$ and $x\in M$ and $\xi^*\in\mathcal{A}_{t,x}$ be a minimal curve in the definition of $u(t,x)$ with $y^*=\xi^*(0)$. Consider the Carath\'eodory equation respect to $L^{\xi^*}$. That is
\begin{equation}\label{eq:Cara3}
	\begin{cases}
		\dot{v}_{\eta}=L(s,\eta,\dot{\eta},u_{\xi^*})+L_u(s,\eta,\dot{\eta},u_{\xi^*})(v_{\eta}-u_{\xi^*}),\quad a.e.\ s\in[0,t]\\
		v_{\eta}(0)=\phi(y^*),
	\end{cases}
\end{equation}
where $\eta\in\Gamma^{0,t}_{y^*,x}$. By solving \eqref{eq:Cara3} we have
\begin{align*}
	v_{\eta}(t)=e^{\int^t_0L_ud\tau}\phi(\eta(0))+\int^t_0e^{\int^t_sL_ud\tau}\{L(s,\eta,\dot{\eta},u_{\xi^*})-u_{\xi^*}L_u\}\ ds.
\end{align*}
Invoking Lemma \ref{equiv1}, we have that
\begin{align*}
	u(t,x)=\inf_{\xi\in\Gamma^{0,t}_{y^*,x}}u_{\xi}(t)=\inf_{\eta\in\Gamma^{0,t}_{y^*,x}}v_{\eta}(t).
\end{align*}
This leads to \eqref{eq:rep3}.
\end{proof}

%
 
\section{representation formula of stationary equation}

In this section, we will study the representation formula of the unique viscosity solution of \eqref{eq:intro_HJs} with $L$ time-independent. Fix $ x\in M$ and $t>0$, denote the set
\begin{align*}
	\mathcal{A}^*_{t,x}=&\,\{\xi\in W^{1,1}([-t,0],M): \xi(0)=x\},\\
	\mathcal{A}^*_{\infty,x}=&\,\{\xi\in W^{1,1}((-\infty,0],M): \xi(0)=x\}.
\end{align*}
Suppose $u$ is the unique viscosity solution of \eqref{eq:intro_HJs}. For any $\xi\in\mathcal{A}^*_{t,x}$ we consider the Carath\'eodory equation
\begin{equation}\label{eq:Caratheodory}
	\begin{cases}
		\dot{u}_{\xi}(s)=L(\xi(s),\dot{\xi}(s),u_{\xi}(s)),\quad \text{a.e.}\ s\in[-t,0],&\\
		u_{\xi}(-t)=u(\xi(-t)).&
	\end{cases}
\end{equation}
We know the viscosity solution $u$ satisfies that property that $u(x)=(\mathbf{T}^t_0u)(x)=u(t,x)$ for all $t\geqslant 0$. Then, we rewrite $u$ as
\begin{equation}\label{Hu_inf}
	\begin{split}
		u(x)=&\,\inf_{\xi\in\mathcal{A}^*_{t,x}}\left\{u(\xi(-t))+\int_{-t}^0L(\xi(s),\dot{\xi}(s),u_{\xi}(s))ds\right\}\\
		=&\,\inf_{y\in M}\left\{u(y)+h_L(t,y,x,u(y))\right\}
	\end{split}
\end{equation}
where $u_{\xi}$ is uniquely determined by \eqref{eq:Caratheodory}. It is known that the infimum in \eqref{Hu_inf} can be achieved.

\begin{The}[Representation formula IV]\label{rep4}
Suppose $L$ satisfies condition \mbox{\rm (L1)-(L3)} and \mbox{\rm (L6)} and $H$ is the associated Hamiltonian, and \eqref{eq:intro_HJs} has a Lipschitz viscosity solution $u(x)$, then the following representation formula holds
\begin{equation}\label{Hu_repre}
u(x)=\inf_{\xi\in\mathcal{A}^*_{\infty,x}}\int_{-\infty}^{0}e^{\int_{s}^{0}L_{u}(\xi,\dot{\xi}, u_{\xi})d\tau}(L(\xi,\dot{\xi},u_\xi)-u_\xi\cdot L_{u}(\xi,\dot{\xi}, u_{\xi}))ds,
\end{equation}
where $u_{\xi}$ satisfies \eqref{eq:Caratheodory} with $u_\xi(0)=u(\xi(0))=u(x)$ for all $t>0$. Moreover, the infimum in \eqref{Hu_repre} can be achieved.
\end{The}

\begin{Rem}
From \eqref{eq:rep2} it is obvious to see that if $L_u$ satisfies a more restricted condition such that $-K\leqslant L_u\leqslant-\delta<0$, then one can see the term $e^{\int_{-t}^0 L_u} u_{\xi}(0)$ in \eqref{eq:rep2} vanishes when $t\to\infty$. However, if only our assumption (L3) is supposed, we need some a priori estimate to ensure the existence of such a positive $\delta$ for the solution of \eqref{eq:intro_HJs}.

For any $x\in M$, if $\xi^*=\xi^*_x\in\mathcal{A}^*_{\infty,x}$ is such a minimizer, we call $\xi^*$ a backward calibrated curve from $x$.
\end{Rem}

\begin{proof}
Under our assumptions, it is well known that the viscosity solution of \eqref{eq:intro_HJs} is unique. Recall that $0$ on the right side of \eqref{eq:intro_HJs} is a critical value.  

Now, let $\xi_t\in\mathcal{A}^*_{t,x}$ be a minimizer for \eqref{Hu_inf}, then the viscosity solution $u$ of \eqref{eq:intro_HJs} has the following representation formula
\begin{equation}\label{eq:cara4}
	\begin{split}
		u(x)&=u(\xi_t(-t))+\int^0_{-t}L\big(\xi_t(s),\dot{\xi}_t(s),u_{\xi_t}(s)\big)	\,ds\\
	&=e^{\int^0_{-t}L_u d\tau}u(\xi_t(-t))+\int^0_{-t}e^{\int^0_s L_u\ d\tau}(L-u_{\xi_t}\cdot L_u)\,ds
	\end{split}
\end{equation}
where $u_{\xi_t}=u(\xi_t(s))$ satisfies the associated Carath\'eodory equation \eqref{eq:Caratheodory} with the initial condition $u_{\xi_t}(-t)=u(\xi_t(-t))$ for all $t\geqslant 0$.

It is clear that, $u_{\xi_t}(s)=u(\xi_t(s))$ and $Du(\xi_t(s))$, $s\in[-t,0]$ are uniformly bounded and $\{(\xi_t, \dot{\xi}_t)\}_{t>0}$ is uniformly bounded (see, for instance, \cite{CCJWY2019}). Invoking Ascoli-Arzela theorem, let $t\rightarrow\infty$, we can find a subsequence $\{\xi_{t_k}\}$ uniformly converges, on any compact subinterval of $(-\infty,0]$, to a Lipschitz curve $\xi:(-\infty,0]\to M$.  Let $u_{\xi}(s):=u(\xi(s)):(-\infty,0]\to\R$, then $u_{\xi_{t_k}}$ uniformly converges to $u_{\xi}$ on any compact interval of $(-\infty,0]$.

Noticing the facts that $\dot{\xi}_k$ and $u_{\xi_k}$ are uniformly bounded,  we can find a $\delta>0$ such that $-K\leqslant L_u(\xi_k,\dot{\xi}_k,u_{\xi_k})\leqslant-\delta<0$. Therefore, by applying the dominated convergence theorem and the equality in \eqref{eq:cara4}, the following representation formula holds
\begin{align*}
	u(x)=\int^0_{-\infty}e^{\int^0_s L_u(\xi,\dot{\xi},u_{\xi})\ d\tau}(L(\xi,\dot{\xi},u_{\xi})-u_{\xi}\cdot L_u(\xi,\dot{\xi},u_{\xi}))\ ds.
\end{align*}
and $u_{\xi}(s)=u(\xi(s))$ for all $s\in(-\infty,0]$. Furthermore,  the representation formula \eqref{Hu_repre} holds. This completes the proof.
\end{proof}

\begin{Cor}[Representation formula V]\label{rep5}
Suppose $L$ satisfies condition \mbox{\rm (L1)-(L6)} and $H$ is the associated Hamiltonian. Then the unique viscosity solution $u$ of \eqref{eq:intro_HJs} has the following representation formula: let $x\in M$ and $\xi^*\in\mathcal{A}^*_{\infty,x}$ be a backward calibrated curve from $x$ with $u_{\xi^*}$ satisfying \eqref{eq:Caratheodory} with respect to $\xi^*$ for all $t>0$, then
\begin{equation}\label{Hu_repre1}
u(x)=\inf_{\xi\in\mathcal{A}^*_{\infty,x}}\int^0_{-\infty}e^{\int^0_s L_u(\xi,\dot{\xi},u_{\xi^*})\ d\tau}(L(\xi,\dot{\xi},u_{\xi^*})-u_{\xi^*}L_u(\xi,\dot{\xi},u_{\xi^*}))\ ds.
\end{equation}
Moreover, the infimum in \eqref{Hu_repre1} can be achieved.
\end{Cor}

\begin{proof}
	The conclusion is direct from Theorem \ref{rep4} and Theorem \ref{rep3}.
\end{proof}

\section{Concluding remarks}

\subsection{Discounted system as an example}

Throughout this section we set
\begin{equation}\label{eq:discount_L}
	L(x,v,u)=L_0(x,v)-\lambda u
\end{equation}
with $\lambda\in\R$, where $L_0$ is a Tonelli Lagrangian on $TM$. Let $H_0$ be the associated Hamiltonian with respect to $L_0$. 

\begin{Cor}
Let $L$ be a discounted Lagrangian defined in \eqref{eq:discount_L} with $\lambda\in\R$ and $H$ be the associated Hamiltonian. Then the unique viscosity solution $u$ of \eqref{eq:HJe} has the following representation:
\begin{align*}
	u(t,x)=\inf_{\xi}\left\{e^{-\lambda t}\phi(\xi(-t))+\int^0_{-t}e^{\lambda s}L_0(\xi(s),\dot{\xi}(s))\ ds\right\}
\end{align*}
where the infimum is taken over the set of the absolutely continuous curve $\xi:[-t,0]\to M$ such that $\xi(0)=x$.
\end{Cor}

\begin{proof}
Applying Proposition \ref{RF_II}, we have
\begin{equation}\label{eq:inf_conv}
	u(t,x)=\inf_{\xi\in\mathcal{A}_{t,x}}\left\{e^{-\lambda t}\phi(\xi(0))+\int^t_0e^{\lambda(s-t)}L_0(\xi(s),\dot{\xi}(s))\ ds\right\}.
\end{equation}
By using the variable-changing $\tau=s-t$ and $\eta(\tau)=\xi(\tau+t)$, we complete our proof.
\end{proof}

A consequence of Theorem \ref{rep4} leads to the following result. See also \cite{DFIZ2016}.

\begin{Cor}
Let $L$ be a discounted Lagrangian defined in \eqref{eq:discount_L} with $\lambda>0$ and $H$ be the associated Hamiltonian. Then the unique viscosity solution $u$ of \eqref{eq:intro_HJs} has the following representation:
\begin{align*}
	u(x)=\inf_{\xi}\left\{\int^0_{-\infty}e^{\lambda s}L_0(\xi(s),\dot{\xi}(s))\ ds\right\}
\end{align*}
where the infimum is taken over the set of the curve $\xi:(-\infty,0]\to M$, which is absolutely continuous on each compact interval of $(-\infty,0]$,such that $\xi(0)=x$.
\end{Cor}

Fix $\lambda\in\R$. From \eqref{eq:inf_conv}, we also have that
\begin{align*}
	u(t,x)=e^{-\lambda t}\inf_{\xi\in\mathcal{A}_{t,x}}\left\{\phi(\xi(0))+\int^t_0e^{\lambda s}L_0(\xi(s),\dot{\xi}(s))\ ds\right\}.
\end{align*}
Set $L^\lambda(t,x,v)=e^{\lambda t}L_0(x,v)$. Then, the associated Hamiltonian has the form $H^{\lambda}(t,x,p)=e^{\lambda t}H_0(x,e^{-\lambda t}p)$. Therefore, $v(t,x)=e^{\lambda t}u(t,x)$ is a viscosity solution of the Hamilton-Jacobi equation
\begin{equation}\label{eq:dis}
	\begin{cases}
	D_tv+e^{\lambda t}H_0(x,e^{-\lambda t}D_xv)=0,& (t,x)\in(0,+\infty)\times M\\
	v(0,x)=\phi(x),& x\in M,
	\end{cases}
\end{equation}
if and only if $u(t,x)$ is a viscosity solution of
\begin{align*}
\begin{cases}
	D_tu+\lambda u+H_0(x,D_xu)=0,& (t,x)\in(0,+\infty)\times M\\
	u(0,x)=\phi(x),& x\in M,
\end{cases}
\end{align*}

Similarly, $u$ is a viscosity solution of the stationary equation
\begin{align*}
	\lambda u(x)+ H_0(x,Du(x))=0, \quad x\in M
\end{align*}
if and only if $v(t,x)=e^{\lambda t}u(x)$ is a viscosity solution of \eqref{eq:dis} with $\phi=u$. 

One can compare with the discussions in \cite{Chen_Cheng_Zhang2018}.

\subsection{Concluding remarks} From the solving-ODE method used previously, we should have more comments on the representation formula for the viscosity solutions of the contact type Hamilton-Jacobi equations. 
\begin{itemize}[--]
	\item Consider
	\begin{equation}
		\left\{
       \begin{aligned}
       &\dot{u}_{\xi}(s)=L(s,\xi(s),\dot{\xi}(s),u_\xi(s))  \quad a.e.\ s\in[0,t]\\
       &u_{\xi}(0)=u_0.\\
       \end{aligned}
       \right.
	\end{equation}
	Observe that
	\begin{equation}\label{eq:split1}
		L(s,\xi(s),\dot{\xi}(s),u_\xi(s))=L(s,\xi(s),\dot{\xi}(s),0)+\widehat{L_u}(s)\cdot u_{\xi}(s)
	\end{equation}
	where
	\begin{align*}
		\widehat{L_u}(s)=\int^1_0L_u(s,\xi(s),\dot{\xi}(s),\lambda u_\xi(s))\ d\lambda,\quad s\in[0,t].
	\end{align*}
	Therefore,
	\begin{align*}
		u_{\xi}(t)=e^{\int^{t}_0\widehat{L_u}\ d\tau}\cdot u_{\xi}(0)+\int^t_0e^{\int^{t}_s\widehat{L_u}\ d\tau}L(s,\xi(s),\dot{\xi}(s),0)\ ds.
	\end{align*}
	This leads to a new representation formula for solutions of \eqref{eq:HJe}. This formula also affords an easy way for the a priori estimate of $u_{\xi}$ which is essential to ensure the existence of solutions for relevant stationary equations.
	\item Replacing \eqref{eq:split1} by 
	\begin{align*}
		L(s,\xi(s),\dot{\xi}(s),u_\xi(s))=&\,\{L(s,\xi(s),\dot{\xi}(s),u_{\xi}(s))-F(s,\xi(s),\dot{\xi}(s),u_{\xi}(s))\cdot u_{\xi}(s)\}\\
		&\,+F(s,\xi(s),\dot{\xi}(s),u_{\xi}(s))\cdot u_{\xi}(s),
	\end{align*}
	where $F$ is an arbitrary $C^1$ function, we obtain that
	\begin{align*}
		u_{\xi}(t)=&\,e^{\int^{t}_0Fd\tau}\cdot u_{\xi}(0)+\int^t_0e^{\int^{t}_sF\ d\tau}\{L(s,\xi(s),\dot{\xi}(s),u_{\xi}(s))-F(s)u_{\xi}(s)\}\ ds\\
		=&\,e^{\int^{t}_0Fd\tau}\cdot u_{\xi}(0)+\int^t_0e^{\int^{t}_sF\ d\tau}\{L(s,\xi(s),\dot{\xi}(s),0)+(\widehat{L_u}(s)-F(s))u_{\xi}(s)\}\ ds
	\end{align*}
	where $F(s):=F(s,\xi(s),\dot{\xi}(s),u_{\xi}(s))$. This leads to another new representation formula for solutions of \eqref{eq:HJe}. But, it is unclear if such a formula have some applications, when 
	\begin{align*}
		\widehat{L_u}-F>0,\quad F<0,
	\end{align*}
	especially to study the solutions of relevant stationary equations when $L_u>0$.
\end{itemize}

\bibliographystyle{abbrv}
\bibliography{mybib}

\end{document}